\documentclass{amsart}
\pdfoutput=1

\usepackage{amsmath,amssymb,amsthm}
\usepackage{amsthm}
\usepackage{graphicx}
\usepackage{float}

\theoremstyle{plain}
\newtheorem{thm}{Theorem}[section]
\newtheorem{lem}[thm]{Lemma}
\newtheorem{prop}[thm]{Proposition}
\newtheorem{cor}[thm]{Corollary}
\newtheorem{que}[thm]{Question}

\theoremstyle{definition}
\newtheorem{defn}[thm]{Definition}

\theoremstyle{remark}
\newtheorem{rem}[thm]{Remark}

\newcommand{\gsw}{\mathcal G^*_\mathrm{weak}}

\newcommand{\Diff}{\mathrm{Diff}}
\newcommand{\intr}{\mathrm{int}}
\newcommand{\AKMR}{Auckly-Kim-Melvin-Ruberman}
\newcommand{\SO}{\mathrm{SO}}
\newcommand{\finset}[1]{{\{1,\dots,#1\}}}

\newenvironment{tightenum}{
  \begin{enumerate}
    \renewcommand{\theenumi}{(\arabic{enumi})}
    \renewcommand{\labelenumi}{\theenumi}
    \setlength{\itemsep}{0pt}
    \setlength{\parskip}{0pt}
}{\end{enumerate}}

\begin{document}

\title{Infinite nonabelian corks}
\author{Hiroto Masuda}
\address{
  Department of Mathematics Faculty of Science and Technology, Keio University
  Yagami Campus: 3-14-1 Hiyoshi, Kohoku-ku, Yokohama, 223-8522, Japan}
\email{hiroto.masuda@keio.jp}

\begin{abstract}
  We construct $G$-corks for any extension $G$ of $\mathbb Z^m$ by any finite subgroup of $\SO(4)$
and weakly equivariant $G$-corks for any extension $G$ of $\mathbb Z^m$ by any finite solvable group.
In particular, this is the first example of $G$-corks for an infinite nonabelian group $G$
and answers a question by Tange.
The construction is a combination of previous results by \AKMR{}, Gompf, and Tange.
Using Gompf's results about exotic $\mathbb R^4$'s, we give an application to construct exotic $\mathbb R^4$'s
whose diffeotopy group contains all poly-cyclic groups.

\end{abstract}

\maketitle

\section{Introduction}
\label{sec:intro}

Understanding all smooth structures on a given topological 4-manifold is a fundamental problem in 4-manifold theory.
In dimension 4, manifolds often have a lot of exotic structures.
For example, there are uncountably many ones on $\mathbb R^4$ (\cite{taubes1987}).
Closed $4$-manifolds can also have infinite ones.
These contrast sharply with facts in other dimensions (cf. \cite[Theorem 1.1.8, 1.1.9]{gompf_stipsicz}).
The richness of exotic structures in dimension $4$ is usually proved
with the result of Freedman about topological $4$-manifolds
and that of Donaldson about smooth $4$-manifolds (cf. \cite{gompf_stipsicz}).

In 4-manifold theory, surgeries to create exotic structures on manifolds are studied actively.
A cork twist is one of them.
A cork is the pair of a compact contractible 4-manifold $C$ and an involution $f$ on its boundary
which does not extend to a diffeomorphism on the whole manifold.
For a closed 4-manifold $X$ and an embedding of $C$ into $X$,
we consider a new manifold $X'$ obtained by cutting $C$ out and regluing it by the boundary diffeomorphism $f$.
This operation is called a cork twist.
The two manifolds $X$ and $X'$ are always homeomorphic (see Theorem \ref{freedman_homeo}).
Curtis-Freedman-Hsiang-Stong and Matveyev proved in \cite{curtis1996} and \cite{matveyev1996}
that for any two simply connected homeomorphic 4-manifolds,
there exists an embedded cork in one of them such that the other is obtained by the cork twist.
We will refer to this as the cork theorem.
There are many studies about corks and their applications from various viewpoints
(cf. \cite{akbulut1991}, \cite{akbulut_yasui2008}, \cite{akbulut_yasui2009},
\cite{yasui2019}, 
\cite{auckly2015}, \cite{yasui2015}).

The notion of a cork is extended to that of a $G$-cork and a weakly equivariant $G$-cork,
where $G$ is a group.
Instead of an involution, a $G$-cork uses a $G$-action on the boundary and 
a weakly equivariant $G$-cork uses a group homomorphism from $G$ to the diffeotopy group
\footnote{
  The diffeotopy group of a manifold $M$ is $\pi_0(\Diff(M))$.
  In other words, it is a mapping class group in the smooth category.}
of the boundary.
Tange first found $\mathbb Z_n$-corks in \cite{tange2017}.
\AKMR{} constructed $G$-corks for any finite subgroup $G$ of $\SO(4)$ 
and weakly equivariant $A$-corks for any finite abelian group $A$
\footnote{The construction in \cite[Theorem B]{auckly2017} can be used for any finite abelian group $A$.}
with an ingenious trick in \cite{auckly2017}.
Gompf constructed $\mathbb Z$-corks in \cite{gompf2017} and
this work was extended to $\mathbb Z^m$-corks by Tange in \cite{tange2016arxiv}.

In spite of these active studies on $G$-corks,
a $G$-cork or a weakly equivariant $G$-cork for an infinite and nonabelian group $G$ has not been found,
as Tange noted in \cite[Question 1.5]{tange2016arxiv}.
We deal with this problem.
The aim of this research is determining large classes of infinite nonabelian groups
for each $G$ in which there exists a $G$-cork or a weakly equivariant $G$-cork.
Additionally, we consider if corks have an effective embedding, that is,
an embedding into a closed 4-manifold whose cork twist yields pairwise homeomorphic but nondiffeomorphic 4-manifolds.

This theme is worth studying for the following reasons.
First, it is interesting from the viewpoint of the diffeomorphism group of a homology $3$-sphere.
The boundary of a $G$-cork (resp. a weakly equivariant $G$-cork) is a homology $3$-sphere and
its diffeomorphism group (resp. its diffeotopy group) contains a subgroup isomorphic to $G$.
Secondly, cork theory is closely related to the diffeotopy group of an exotic $\mathbb R^4$ by Gompf (\cite{gompf2018}).
Progress about cork theory may bring new results about exotic $\mathbb R^4$'s.
Indeed, (the proof of) our main theorem has an interesting application about them (Theorem \ref{thm:app:pc}).
Thirdly, Tange pointed out that cork twists define a nontrivial structure on the set of all smooth structures on a given $4$-manifold.
He showed that a $\mathbb Z$-cork analog of the cork theorem is false in \cite{tange2016nonext} (cf. \cite{yasui2019}).
Moreover, Melvin-Schwartz showed that a $\mathbb Z_n$-cork analog is true in \cite{melvin_schwartz}.
Thus we should consider $G$-corks and weakly equivariant $G$-corks for an infinite group $G$ to reveal this structure.

We state our main theorem and an outline of its proof.
First, we remark that we can not extend the trick in \cite{auckly2017} directly to the case of infinite groups.
Finiteness of a group is inherent in it because the resulting $G$-cork is the boundary sum of $|G|$ copies of a cork.
However, substituting Gompf's cork for the cork in the trick gives extra $\mathbb Z$-actions on the resulting $G$-cork.
We analyze the action generated by these $G$- and extra $\mathbb Z$-actions in detail
and find a class of infinite nonabelian groups $G$ with $G$-corks. 
Moreover, we see that we can iterate our idea in the case of weakly equivariant $G$-corks.
With group theoretical considerations,
it shows that there exist weakly equivariant $G$-corks for $G$ in a quite large class of nonabelian groups.

\begin{thm}
  \label{thm:main}
Let $G$ be any extension of $\mathbb Z^m$ by any finite group $H$ for any integer $m\ge0$.
\begin{tightenum}
\item\label{thm:main:eqv}
  If $H$ is a (finite) subgroup of $\SO(4)$, there exists a $G$-cork.
\item\label{thm:main:eqvgen}
  There exists a generalized $G$-cork.
\item\label{thm:main:weak}
  If $H$ is solvable, there exists a weakly equivariant $G$-cork.
\item\label{thm:main:weakStein}
  If $H$ is solvable, there exists a Stein weakly equivariant $H$-cork.
\end{tightenum}
\end{thm}

\noindent A generalized $G$-cork is a $G$-cork which can be non-contractible.
The term `extension' means that there exists an exact sequence $1\rightarrow\mathbb Z^m\rightarrow G\rightarrow H\rightarrow1$.

We add a few remarks to the theorem.
All corks in the theorem have an effective embedding into a closed $4$-manifold.
The corks in \ref{thm:main:eqv} and \ref{thm:main:weak} (resp. \ref{thm:main:weakStein})
are boundary sums of copies of Gompf's cork (resp. Akbulut's cork).
Since whether Gompf's corks admit Stein structures is unsolved (\cite[Remarks (b)]{gompf2017_inf_cork_hdl}),
we do not know whether the corks in \ref{thm:main:eqv} and \ref{thm:main:weak} are Stein.

The theorem implies the following two facts.
First, we have a $\mathbb Z_2*\mathbb Z_2$-cork
since $\mathbb Z_2*\mathbb Z_2$ is an extension of $\mathbb Z$ by $\mathbb Z_2$.
This is the first example of $G$-corks for a free product $G$ of two nontrivial groups.
Secondly, it follows that for any finite solvable group $G$,
there exists a homology $3$-sphere admitting a Stein fillable contact structure
whose diffeotopy group contains $G$. 

We explain how cork theory is applied to exotic $\mathbb R^4$'s and what our theorem implies.
There is an open analog
of the cork theorem,
which twists an embedded exotic $\mathbb R^4$ in a manifold to change its smooth structure instead of a cork
(cf. \cite[Theorem 9.2.18 and 9.3.1]{gompf_stipsicz}).
Using this similarity of corks and exotic $\mathbb R^4$'s,
Gompf imported the trick in \cite{auckly2017} to the context of exotic $\mathbb R^4$'s
and constructed uncountably many exotic $\mathbb R^4$'s
whose diffeotopy group contains interesting subgroups,
such as the direct product of countable copies of $\mathbb Q$ and the countably generated free group (\cite{gompf2018}).
However, groups with torsion elements were not dealt with very much.
It is natural to ask what finite groups diffeotopy groups of exotic $\mathbb R^4$'s can contain.
Combining Gompf's idea with our theorem, we get the following theorem,
which gives a partial answer to this problem.

\begin{thm}
  \label{thm:app:pc}
  There exists an exotic $\mathbb R^4$ whose diffeotopy group contains any poly-cyclic group.
\end{thm}

\noindent For the definition of a poly-cyclic group, see Definition \ref{def_of_pc}.
We remark that all finite solvable groups are poly-cyclic by the definition.

This paper consists of 5 sections.
In Section \ref{sec:pre}, we first see basic facts about a wreath product of groups.
It plays a key role in our proof of Theorem \ref{thm:main}.
We then check the precise definitions of corks and state previous results needed later.
We prove Theorem \ref{thm:main} \ref{thm:main:eqv} and \ref{thm:main:eqvgen} in Section \ref{sec:cork}
and Theorem \ref{thm:main} \ref{thm:main:weak} and \ref{thm:main:weakStein} in Section \ref{sec:weak_corks}.
In both sections, we first show them in the case of wreath products of certain groups
and see that it is in fact equivalent to the general one.
In Section \ref{sec:diff}, we prove Theorem \ref{thm:app:pc}.

\section{Preliminaries}
\label{sec:pre}
In this section, we first review the definition of a wreath product of groups and its properties.
We then define some terminologies about corks and state some previous results needed later.

Throughout this paper, all manifolds are smooth.
A closed (resp. compact) manifold is a manifold
which is compact and without a boundary (resp. with a nonempty boundary).
A diffeomorphism is orientation preserving unless otherwise stated.
A diffeotopy is a smooth ambient isotopy.
For a manifold $M$, $\Diff(M)$ denotes the diffeomorphism group of $M$
and $\mathcal D(M)$ denotes the diffeotopy group of $M$.
The symbol $1$ or $1_G$ denotes the unit in a group $G$.

\subsection{Wreath products}
A wreath product is a well-known operation in group theory.
For more details, see \cite{suzuki}.
This operation will be used repeatedly in our construction.

Let $H$ and $G$ be groups.
The symbol $H^G$ denotes the set of all maps from $G$ to $H$.
It has the natural group structure defined by the pointwise multiplication.
For any $g\in G$ and $F\in H^G$,
we define ${}^gF\in H^G$ by ${}^gF(x):=F(g^{-1}x)$ for all $x\in G$.
This determines a left action of $G$ on $H^G$.
\begin{defn}
  The wreath product $H\wr G$ of $H$ by $G$ is the semidirect product $H^G\rtimes G$ of $H^G$ by $G$ 
  with respect to the left action ${}^gF$.

\end{defn}
\noindent Note that the multiplication is defined by $(F_1,g_1)(F_2,g_2):=(F_1({}^{g_1}F_2),g_1g_2)$
for all $(F_1,g_1)$, $(F_2,g_2)$ $\in H\wr G$.
The group $H\wr G$ is an extension of $H^G$ by $G$
since $H^G\times\{1\}$ is a normal subgroup with $H\wr G/H^G\times\{1\}\cong G$.
If $G'\subset G$ and $H'\subset H$ are subgroups,
$H'\wr G'$ can be embedded into $H\wr G$.
\footnote{An embedding of a group means an injective homomorphism.}

The following theorem states an important property of a wreath product.
\begin{thm}[{\cite[Theorem 10.9]{suzuki}}]
  \label{kk_thm}
  Let $N\triangleleft G$ be a normal subgroup. The group $G$ can be embedded into the group $N\wr(G/N)$.
\end{thm}

\noindent This implies the following corollary.
For the sake of simplicity, we omit parentheses in an iterated wreath product,
i.e., we write simply
\[
H_r\wr H_{r-1}\wr\dots\wr H_1=\left(\cdots\left(H_r\wr H_{r-1}\right)\wr H_{r-2}\cdots\right)\wr H_1,
\]
although a wreath product is not associative.
\begin{cor}[{\cite[Corollary of Theorem 10.9]{suzuki}}]
  \label{kk_cor}
  Let $G=G_0\triangleright G_1\triangleright\dots\triangleright G_r={1}$ be a subnormal series and
  let $H_i:=G_{i-1}/G_i$ be its quotient group.
  The group G can be embedded into the iterated wreath product
  $H_r\wr H_{r-1}\wr\dots\wr H_1$.
\end{cor}

\noindent In particular, we have the following.
\begin{cor}[{\cite[the paragraph below Corollary of Theorem 10.9]{suzuki}}]
  \label{kk_cor_fs}
  For any finite solvable group $G$,
  there exist prime numbers $p_1,\dots,p_r$ such that
  $G$ can be embedded into $\mathbb Z_{p_1}\wr\mathbb Z_{p_2}\wr\dots\wr\mathbb Z_{p_r}$.
\end{cor}

\noindent Since $\mathbb Z_{n_1}\wr\mathbb Z_{n_2}\wr\dots\wr\mathbb Z_{n_r}$ is an extension of
$(\mathbb Z_{n_1}\wr\mathbb Z_{n_2}\wr\dots\wr\mathbb Z_{n_{r-1}})^{n_r}$ by $\mathbb Z_{n_r}$,
an inductive argument shows that $\mathbb Z_{n_1}\wr\mathbb Z_{n_2}\wr\dots\wr\mathbb Z_{n_r}$ is also solvable.

Lastly, we state the following slightly complicated proposition.
This will be used in subsequent sections.
\begin{prop}
  \label{homom_prop}
  Let $G,N$ be groups, let $H$ be a finite group, let $\phi\colon H\rightarrow G$ be a homomorphism,
  and let $\psi_h\colon N\rightarrow G$ be homomorphisms indexed by elements $h\in H$.
  We assume that for any $x\in N$ and $g,h\in H$, we have $\phi(h)\psi_g(x)\phi(h^{-1})=\psi_{hg}(x)$.
  Moreover, we assume that for any $x,y\in N$ and $g,h\in H$ with $g\neq h$,
  we have $\psi_g(x)\psi_h(y)=\psi_h(y)\psi_g(x)$.
  Then the map $\omega\colon N\wr H\rightarrow G$
  defined by $\omega(F,h):=\left(\prod_{g\in H}\psi_g(F(g))\right)\phi(h)$ is a homomorphism.
\end{prop}
Note that $\omega$ is well-defined by the assumptions,
that is, the order of multiplication of $\prod_{g}\psi_g(F(g))$ does not affect the result.
An straightforward argument shows the proposition.

\subsection{Corks}
We first review terminologies related to corks.
We follow the conventions in \cite{auckly2017}.
\footnote{
  The definition of a cork in this section is slightly different from that in Section \ref{sec:intro}.
  In this section, the boundary diffeomorphism of a cork does not have to be an involution.
  In Section \ref{sec:intro}, we followed \cite{akbulut_yasui2009} except for the Stein condition
  for the sake of the explanation.}
Note that there are different ones
(cf. \cite{akbulut_yasui2008}).

Let $C$ be a compact contractible 4-manifold.
A \textit{cork} $(C,f)$ is the pair of $C$ and a diffeomorphism $f$ on $\partial C$.
Furthermore, let $G$ be a group and let $\phi$ be a $G$-action on $\partial C$.
We consider $\phi$ as a homomorphism $\phi\colon G\rightarrow\Diff(\partial C)$.
The pair $(C,\phi)$ is called a \textit{$G$-cork}
if $\phi$ maps any nontrivial element $g\in G\setminus\{1\}$ to a diffeomorphism on $\partial C$
which does not extend over $C$.
Similarly, let $\psi\colon G\rightarrow\mathcal D(\partial C)$ be a homomorphism.
The pair $(C,\psi)$ is called a \textit{weakly equivariant $G$-cork}
if a representative of $\psi(g)$ does not extend over $C$ for any nontrivial element $g\in G\setminus\{1\}$ as above.
This is well-defined.

Let $C$ be as above and
let $e\colon C\rightarrow X$ be an embedding into a closed 4-manifold $X$.
Assume that $(C,f)$ is a cork.
Cutting $e(C)$ in $X$ and regluing it by $f$,
we get the new manifold $X_f^e:=(X-\intr C)\cup_f C$.
This operation is called a \textit{cork twist}.
The manifold $X_f^e$ is always homeomorphic to $X$ by the following theorem.
\begin{thm}
  \label{freedman_homeo}
  Let $C$ be a compact contractible 4-manifold and let $f$ be a diffeomorphism on $\partial C$.
  Then $f$ extends over $C$ homeomorphically.
\end{thm}
\noindent A short proof of this theorem using Freedman's results (\cite{freedman1982})
can be found in \cite[Remarks (a)]{gompf2017_inf_cork_hdl}.
If $X_f^e$ is not diffeomorphic to $X$, $e$ is called \textit{effective}.
Next, assume that $(C,\phi)$ is a $G$-cork.
Identifying $G$ and $\phi(G)$, we write simply $X_g^e$ instead of $X_{\phi(g)}^e$.
(Note that $\phi$ is always injective by the definition.)
If the manifolds $X_g^e$ for $g\in G$ are pairwise nondiffeomorphic, $e$ is called \textit{$G$-effective}.
Lastly, assume that $(C,\psi)$ is a weakly equivariant $G$-cork.
For any $g\in G$, choose a representative $f$ of $\psi(g)$.
The diffeomorphism type of $X_f^e$ does not depend on the choice of $f$.
It is denoted by $X_g^e$.
A \textit{$G$-effective embedding} is defined in the same way.

Next, we define a \textit{generalized $G$-cork}
\footnote{We can also define a \textit{generalized cork} and a \textit{generalized weakly equivariant $G$-cork}
but we do not use them.}.
This term was used by Tange in \cite{tange2015}.
Let $(C,\phi)$ be a $G$-cork.
If we allow $C$ to be non-contractible
and assume that $C$ is connected
\footnote{Tange does not assume $C$ to be connected explicitly in \cite{tange2015}.}
and that all the diffeomorphisms $\phi(g)$ on $\partial C$ for $g\in G$ extend over $C$ homeomorphically,
we call $(C,\phi)$ a \textit{generalized $G$-cork}.
A \textit{$G$-effective embedding} of a generalized $G$-cork is defined in the same way.

We will use the following results about disjoint embeddings of a cork and a $\mathbb Z$-cork.
Hereafter, for integers $r,s$, and $m$ with $r,s>0$ and $m\neq0$,
$C(r,s;m)$ denotes Gompf's $\mathbb Z$-cork constructed in \cite{gompf2017}
and $f$ denotes the generator of the $\mathbb Z$-action on $\partial C(r,s;m)$.

\begin{thm}[\AKMR{} {\cite[Lemma 2.2]{auckly2017}}]\label{emb_of_2_cork}
  For any $n\ge 1$, there exist a closed 4-manifold $X$, a cork $(\mathbb S,\sigma)$,
  and $n$ disjoint embeddings $s_i:\mathbb S\rightarrow X$ for $1\le i\le n$ such that
  $X_\sigma^{s_i}$ and $X_\sigma^{s_j}$ are nondiffeomorphic for any $i,j$ with $i\neq j$.
  Moreover, we can take $\mathbb S$ to be diffeomorphic to the boundary sum of copies of Akbulut's cork
  and thus Stein.
\end{thm}

\begin{thm}[Tange {\cite[Theorem 1 and its proof]{tange2016arxiv}}]
  \label{emb_of_z_cork}
  Fix integers $r,s,m$ with $r,s>0$ and $m\neq0$.
  For any $n\ge 1$, there exist a closed 4-manifold $X$ and
  $n$ disjoint embeddings $s_i:C(r,s;m)\rightarrow X$ for $1\le i\le n$ such that
  $X_{k_1,\dots,k_n}^{s_1,\dots,s_n}$ and $X_{l_1,\dots,l_n}^{s_1,\dots,s_n}$ are nondiffeomorphic
  for any $(k_1,\dots,k_n)$, $(l_1,\dots,l_n)\in\mathbb Z^n$ with $(k_1,\dots,k_n)\neq(l_1,\dots,l_n)$.
  Consequently, we obtain a $\mathbb Z^n$-cork by joining them with 1-handles.
\end{thm}

\noindent For the definition of Akbulut's cork, see \cite[Section 9.3]{gompf_stipsicz}.
Note that $X_{k_1,\dots,k_n}^{s_1,\dots,s_n}$ is the manifold obtained by twisting each $s_i(C(r,s;m))$ by $k_i$.
In Theorem \ref{emb_of_2_cork}, we can assume that $(\mathbb S,\sigma)$ is a $\mathbb Z_2$-cork
as they mentioned in \cite{auckly2017} but we do not need this fact.

\section{Constructing $G$-corks}
\label{sec:cork}
In this section, we prove the following theorem
equivalent to Theorem \ref{thm:main} \ref{thm:main:eqv} and \ref{thm:main:eqvgen}.

\begin{thm}
  \label{thm:wreath:eqv}
  Let $m\ge1$ be an integer.
  \begin{tightenum}
  \item\label{thm:wreath:eqv:eqv}
    For any finite subgroup $H\subset\SO(4)$, there exists a $\mathbb Z^m\wr H$-cork.
  \item\label{thm:wreath:eqv:eqvgen}
    For any finite group $H$, there exists a generalized $\mathbb Z^m\wr H$-cork.
  \end{tightenum}
\end{thm}

\begin{prop}
  Theorem \ref{thm:wreath:eqv} \ref{thm:wreath:eqv:eqv} and \ref{thm:wreath:eqv:eqvgen}
  are equivalent to Theorem \ref{thm:main} \ref{thm:main:eqv} and \ref{thm:main:eqvgen} respectively.
\end{prop}

\begin{proof}
  Theorem \ref{kk_thm} and the fact that $\mathbb Z^m\wr H$ is an extension of $\mathbb Z^{m|H|}$ by $H$ imply the proposition.
\end{proof}

We prove \ref{thm:wreath:eqv:eqv} and \ref{thm:wreath:eqv:eqvgen} in the same way.
The main idea is substituting Gompf's $\mathbb Z$-cork
for a $\mathbb Z_2$-cork in \cite[the proof of Theorem A]{auckly2017}.
Recall that $C(r,s;m)$ denotes Gompf's cork and $f$ denotes the generator of the $\mathbb Z$-action
(see Section \ref{sec:pre}).
\begin{proof}[Proof of Theorem \ref{thm:wreath:eqv} \ref{thm:wreath:eqv:eqv}.]
  Fix $\bar r,\bar s,\bar m$ with $\bar r,\bar s>0$ and $\bar m\neq0$ and let $C:=C(\bar r,\bar s;\bar m)$.
  Let $X$ be a closed 4-manifold and $s{(h,i)}:C\rightarrow X$ be disjoint embeddings indexed by $(h,i)\in H\times\finset{m}$.
  For $F\in\mathbb Z^{H\times\finset{m}}$, we use $X(F)$ to denote the manifold
  obtained by twisting each cork $s(h,i)(C)$ by $F(h,i)\in\mathbb Z$.
  By Theorem \ref{emb_of_z_cork}, we can assume that two manifolds $X(F)$ and $X(G)$ are nondiffeomorphic if $F\neq G$.

  Next, we use the trick in \cite{auckly2017}.
  Let $\alpha:B^4\rightarrow X$ be an embedding of a 4-ball which is disjoint from $s{(h,i)}(C)$.
  Choose a point $p\in\partial B^4$ such that $h\cdot p\neq h'\cdot p$ for any $h,h'\in H$ with $h\neq h'$.
  Now $H$ acts on $B^4$ by linear transformations.
  Let $x,y\in\partial C$ be fixed points of $f$.
  We connect $\alpha(h\cdot p)$ with $s{(h,1)}(x)$, $s{(h,1)}(y)$ with $s{(h,2)}(x)$,
  ..., and $s{(h,m-1)}(y)$ with $s{(h,m)}(x)$ for all $h\in H$ by $1$-handles.
  The resulting manifold embedded in $X$ is denoted by $\overline{\mathbb T}$.
  See Figure \ref{fig:boundary_sum}.

  \begin{figure}
  \centering
  \includegraphics[width=9cm]{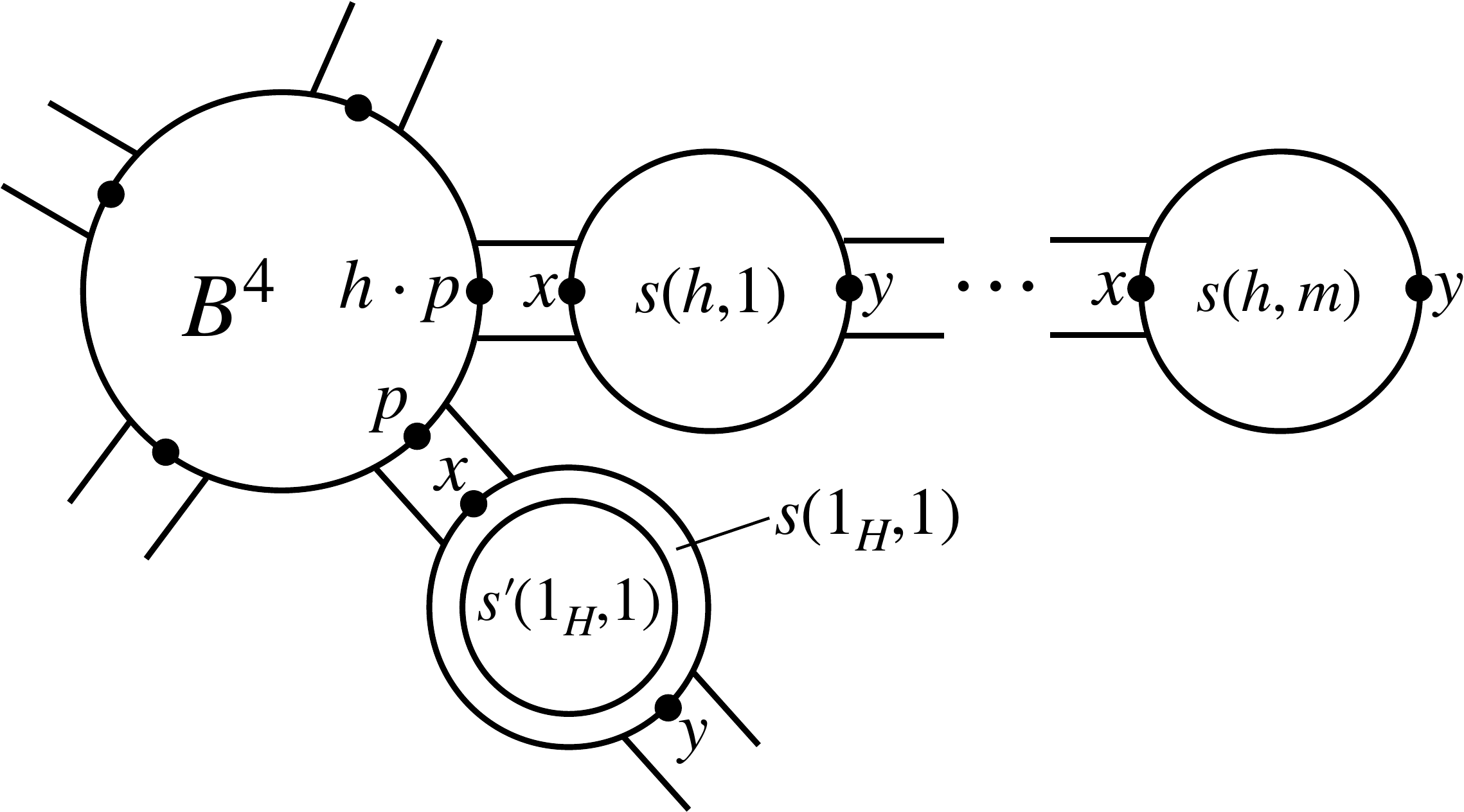}
  \caption{The manifold $\overline{\mathbb T}$.}
  \label{fig:boundary_sum}
  \end{figure}

  Let $s'(1_H,1):C\rightarrow X$ be an embedding obtained by shrinking $s{(1_H,1)}$.
  \footnote{This is also part of the trick in \cite{auckly2017}.}
  To be more precise, $s'(1_H,1)$ is the composition of $s{(1_H,1)}$
  and a shrinking map $C\rightarrow\intr C$ isotopic to the identity on $C$.
  We define $\mathbb T:=\overline{\mathbb T}^{s'(1_H,1)}_1$.
  It is embedded in $Y:=X^{s'(1_H,1)}_1\cong X(\delta_{(1_H,1)})$ naturally,
  where $\delta_{(h,j)}$ denotes the element in $\mathbb Z^{H\times\finset{m}}$
  whose value is 1 at $(h,j)$ and 0 otherwise.
  Let $\iota$ be the embedding of $\mathbb T$ into $Y$.

  We define actions on $\partial\mathbb T$.
  By its construction, $\overline{\mathbb T}$ has the $H$-action
  which is the linear action on $B^4$ 
  and extended on Gompf's corks and 1-handles trivially.
  Since we have $\partial\mathbb T=\partial\overline{\mathbb T}$,
  the restriction of the action to it defines a boundary action 
  $\phi:H\rightarrow\Diff(\partial\mathbb T)$ of $\mathbb T$.
  Next, we define $\mathbb Z^m$-actions $\psi_h$ on $\partial\mathbb T$ indexed by $h\in H$.
  For any $n=(n_1,\dots,n_m)\in\mathbb Z^m$,
  let $\psi_h(n)$ be the diffeomorphism which is $f^{2n_i}$ on the support of $f$
  \footnote{This denotes the closure of the set $\{x\in\partial C\mid f(x)\neq x\}$.}
  in the boundary of each Gompf's cork $s(h,i)(C)$ and
  the identity on the rest.
  Note that the exponent of the power is doubled $n_i$.
  The actions $\phi$ and $\psi_h$ for $h\in H$ satisfy the assumptions of Proposition \ref{homom_prop}
  and we get the homomorphism $\omega:\mathbb Z^m\wr H\rightarrow\Diff(\partial\mathbb T)$.

  Lastly, we twist $\mathbb T$ in $Y$.
  For any $(F,h)\in\mathbb Z^m\wr H$, we have $Y_{(F,h)}^\iota\cong X(2F+\delta_{(h,1)})$.
  Therefore, they are pairwise nondiffeomorphic and $\iota$ is $\mathbb Z^m\wr H$-effective.
\end{proof}

\begin{proof}[Proof of Theorem \ref{thm:wreath:eqv} \ref{thm:wreath:eqv:eqvgen}]
  Let $s(h,i)$ and $X$ be as above.
  We connect $s(h,i)(C)$ in the following way.
  We choose $z_h,z_h'\in\partial C$ indexed by $h\in H\setminus\{1\}$ and $x,y\in\partial C$
  in the fixed points of $f$.
  For each $h\in H$ and $g\in H\setminus\{1\}$, we attach a 1-handle $\mathfrak h(h,g)$
  to $s(h,1)(z_g)$ and $s(hg,1)(z_{g^{-1}}')$.
  Thus, $s(h,1)(C)$ and $s(hg,1)(C)$ are connected by two 1-handles $\mathfrak h(h,g)$ and $\mathfrak h(hg,g^{-1})$.
  Then we connect $s(h,i)(x)$ with $s(h,i+1)(y)$ by a 1-handle $\mathfrak h'(h,i)$ for any $i\in\finset{m-1}$.
  The resulting manifold $\overline{\mathbb T}$ is the Cayley graph of $H$
  whose vertices are boundary sums of copies of $C$ and edges are doubled 1-handles.
  Note that $\overline{\mathbb T}$ is connected.

  Let $s'(1_H,1)$ be a shrunk $s(1_H,1)$ and
  define $\mathbb T:=\overline{\mathbb T}^{s'(1_H,1)}_1$ and $Y:=X^{s'(1_H,1)}_1$.
  The manifold $\mathbb T$ is embedded in $Y$.
  Let $\phi$ be the $H$-action on $\overline{\mathbb T}$
  defined by $\phi(x)s(h,i)(p):=s(xh,i)(p)$ for any $x,h\in H$, $i\in\finset{m}$, and $p\in C$
  on each Gompf's cork
  and extended over 1-handles.
  Thus, we have $\phi(x)(\mathfrak h(h,g))=\mathfrak h(xh,g)$ and $\phi(x)(\mathfrak h'(h,i))=\mathfrak h'(xh,i)$,
  and these are the identities.
  We restrict it to $\partial\overline{\mathbb T}=\partial\mathbb T$.
  Let $\psi_h$ for $h\in H$ be the $\mathbb Z^m$-actions on $\partial\mathbb T$ same as the above proof.
  By Proposition \ref{homom_prop}, we get the homomorphism
  $\omega:\mathbb Z^m\wr H\rightarrow \Diff(\partial\mathbb T)$.

  For any $(F,h)\in\mathbb Z^m\wr H$,
  $\omega(F,h)$ extends homeomorphically over $C$
  since the diffeomorphism $f$ on $\partial C$ extends homeomorphically over $C$ by Theorem \ref{freedman_homeo}.
  We have $Y^\iota_{(F,h)}\cong X(2F+\delta_{(h,1)})$,
  where $\iota:\mathbb T\rightarrow Y$ is the inclusion map.
  It shows that $\iota$ is $\mathbb Z^m\wr H$-effective.
\end{proof}

\begin{que}
  Does there exists a simply connected generalized $H$-cork for any finite group $H$?
\end{que}

\section{Constructing weakly equivariant $G$-corks}
\label{sec:weak_corks}
In this section, we prove the following theorem
equivalent to Theorem \ref{thm:main} \ref{thm:main:weak} and \ref{thm:main:weakStein}.

  \label{wreath_weak_cork}

  \label{wreath_weak_stein_cork}

\begin{thm}
  \label{thm:wreath:weak}
  Let $m\ge1$ and $n_1,\dots,n_r\ge2$ be any integers.
  \begin{tightenum}
  \item\label{thm:wreath:weak:weak}
     There exists a weakly equivariant
     $\mathbb Z^m\wr\left(\mathbb Z_{n_1}\wr\mathbb Z_{n_2}\wr\dots\wr\mathbb Z_{n_r}\right)$-cork.
   \item\label{thm:wreath:weak:Stein}
     There exists a weakly equivariant Stein
     $\mathbb Z_{n_1}\wr\mathbb Z_{n_2}\wr\dots\wr\mathbb Z_{n_r}$-cork.
  \end{tightenum}
\end{thm}
\noindent Recall that we omit parentheses in an iterated wreath product, see Section 2.

\begin{prop}
  Theorem \ref{thm:wreath:weak} \ref{thm:wreath:weak:weak} and \ref{thm:wreath:weak:Stein} are equivalent to
  Theorem \ref{thm:main} \ref{thm:main:weak} and \ref{thm:main:weakStein} respectively.
\end{prop}

\begin{proof}
  Theorem \ref{kk_thm} and Corollary \ref{kk_cor_fs} imply the proposition.
\end{proof}

We use several lemmas to prove them.
First, we define a \textit{block}, which will be used throughout this section.

\begin{defn}
  A block $\Phi=(A;e_1,\dots,e_n;\epsilon)$ is the tuple of a compact 4-manifold $A$ diffeomorphic to $B^4$ and
  disjoint embeddings $e_1,\dots,e_n,\epsilon:B^3\rightarrow\partial A$.
\end{defn}
\noindent
Although $A$ is diffeomorphic to $B^4$,
it becomes useful later to use distinguishable symbols for $B^4$ of blocks.
This is the reason why we include $A$ in the above definition.

Let $f$ be a diffeomorphism on $A$.
If we have $f\circ\epsilon=\epsilon$ and there exists $\sigma\in\mathfrak S_n$
such that $f\circ e_i=e_{\sigma(i)}$ for all $i\in\finset{n}$,
we call $f$ a \textit{diffeomorphism of the block} $\Phi$.
The set of all diffeomorphisms of $\Phi$ is denoted by $\Diff(\Phi)$.
The natural homomorphism
$\Diff(\Phi)\rightarrow\mathfrak S_n$
defined by $f\mapsto\sigma$ is denoted by \ $\widehat\cdot $\ , i.e., $\widehat f=\sigma$.
Two diffeomorphisms $f$ and $g$ of $\Phi$ are called \textit{diffeotopic over} $\Phi$
if there exists a diffeotopy between $f$ and $g$
consisting of diffeomorphisms
in $\Diff(\Phi)$.
The set of all diffeotopy classes over $\Phi$ is denoted by $\mathcal D(\Phi)$.
The homomorphism \ $\widehat\cdot $ \ factors through the quotient map
$\Diff(\Phi)\rightarrow\mathcal D(\Phi)$
and we use the same symbol \ $\widehat\cdot :\mathcal D(\Phi)\rightarrow\mathfrak S_n$.

\begin{defn}
Let $G$ be a group.
For a homomorphism $\phi:G\rightarrow\mathcal D(\Phi)$,
we define the following two properties:
\begin{enumerate}
  \setlength{\parskip}{0cm} 
  \setlength{\itemsep}{0cm} 
\item[(P1)] The map $\widehat\phi$ is injective.
\item[(P2)] There exists $p\in\{1,\dots,n\}$ such that $\widehat\phi(g)p$ are pairwise different for all $g\in G$.
\end{enumerate}
\end{defn}

\begin{rem}
  The above definitions are inspired by the definition of $\mathcal G^*$ in \cite[Section 4]{gompf2018}. 
\end{rem}

\begin{lem}
  \label{weak_cork_lem}
  Let $G$ be a group.
  Assume that there exist a block $\Phi=(A;$ $e_1,\dots,e_n;$ $\epsilon)$
  and a homomorphism $\phi:G\rightarrow\mathcal D(\Phi)$ with (P2).
  Then there exists a weakly equivariant $\mathbb Z^m\wr G$-cork for any $m\ge1$
  and a weakly equivariant Stein $G$-cork.
\end{lem}
\noindent Though the proof is similar to that of Theorem \ref{thm:wreath:eqv},
we describe it in detail.
Recall again that $C(r,s;m)$ denotes Gompf's cork and $f$ denotes the generator of the $\mathbb Z$-action (see Section \ref{sec:pre}).
\begin{proof}
  First, we prove the former. Fix integers $\bar r$,$\bar s$, and $\bar m$ with $\bar r,\bar s>0$ and $\bar m\neq0$,
  and let $C:=C(\bar r,\bar s;\bar m)$.

  Let $X$ be a closed 4-manifold and let $s(i,j):C\rightarrow X$ be disjoint embeddings
  indexed by $(i,j)\in\finset{n}\times\finset{m}$.
  For $F\in\mathbb Z^{\finset{n}\times\finset{m}}$, we use $X(F)$ to denote the manifold
  obtained by twisting each $s(i,j)(C)$ by $F(i,j)\in\mathbb Z$.
  By Theorem \ref{emb_of_z_cork},
  we can assume that two manifolds $X(F)$ and $X(G)$ are nondiffeomorphic if $F\neq G$.

  Next, we join them via embedded $A$.
  Let $\alpha:A\rightarrow X$ be an embedding of $A$ into $X$ disjoint from $s(i,j)(C)$.
  Choose two fixed points $x,y\in \partial C$ of $f$.
  We connect $\alpha\circ e_i(0)$ with $s(i,1)(x)$,
  $s(i,1)(y)$ with $s(i,2)(x)$, ..., and $s(i,m-1)(y)$ with $s(i,m)(x)$ by 1-handles for each $i\in\finset{n}$.
  The resulting submanifold in $X$ is denoted by $\overline{\mathbb T}$.
  See Figure \ref{fig:boundary_sum_weak}.

  \begin{figure}
  \centering
  \includegraphics[width=8cm]{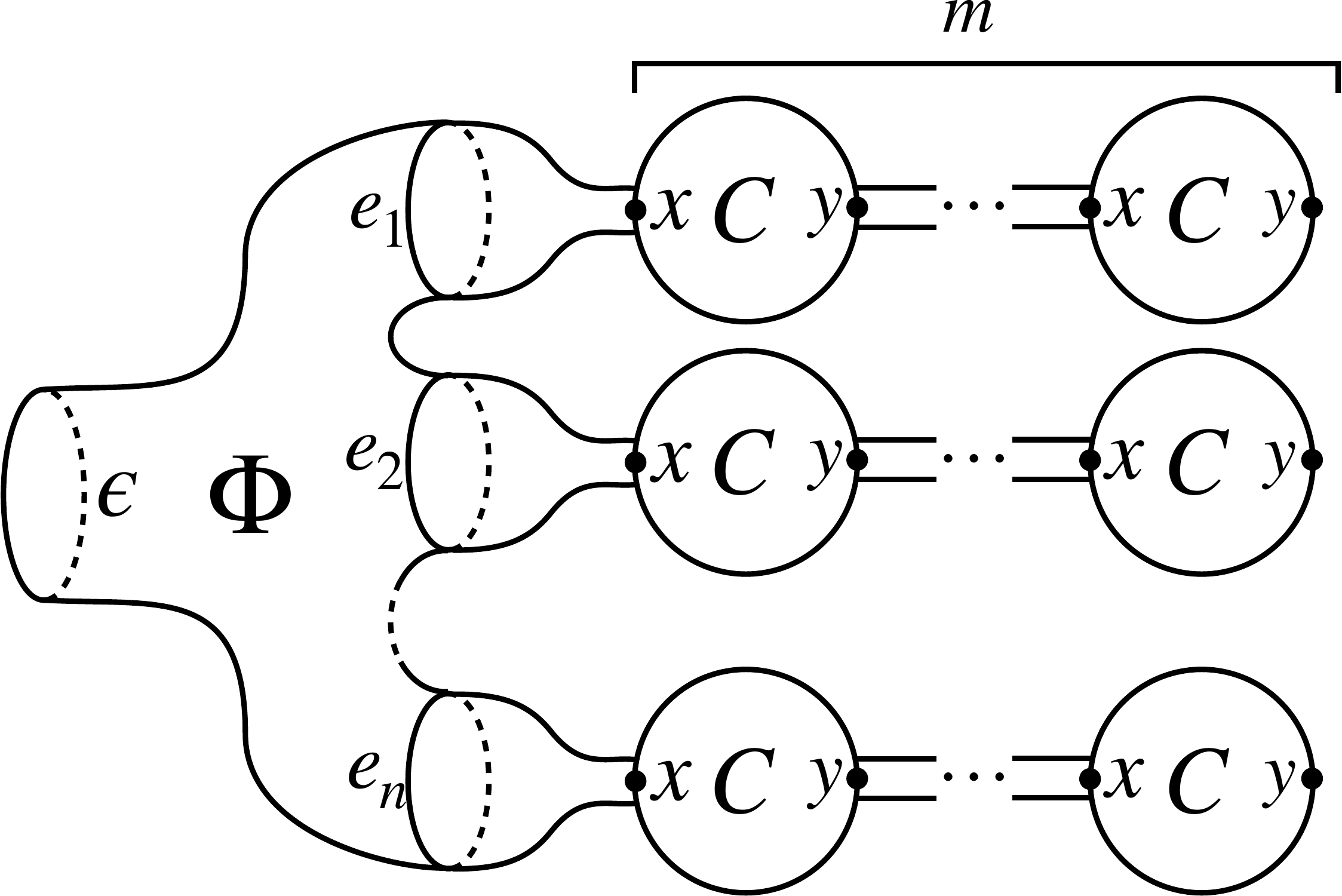}
  \caption{The manifold $\overline{\mathbb T}$.}
  \label{fig:boundary_sum_weak}
  \end{figure}

  Let $p\in\finset{n}$ be a point as in (P2).
  We use the trick in \cite{auckly2017}.
  Let $s'(p,1)$ be a shrunk $s(p,1)$ as in the proof of Theorem \ref{thm:wreath:eqv}.
  We define $\mathbb T:=\overline{\mathbb T}^{s'(p,1)}_1$ and $Y:=X^{s'(p,1)}_1$.
  The manifold $\mathbb T$ is embedded in $Y$ naturally.

  We define actions on $\partial\mathbb T$.
  Let $\phi^e:G\rightarrow\mathcal D(\overline{\mathbb T})$ be the homomorphism
  which is $\phi$ on $\alpha(A)$ and extended over $\overline{\mathbb T}$
  to be the permutation of components of $\overline{\mathbb T}-\alpha(A)$ by $\widehat\phi$.
  We restrict $\phi^e$ to the boundary $\partial\overline{\mathbb T}=\partial\mathbb T$
  and get the homomorphism $\phi^e:G\rightarrow\mathcal D(\partial\mathbb T)$;
  we use the same symbol $\phi^e$.
  Next, fix $g\in G$ and let $i:=\widehat\phi(g)p$.
  For any $k=(k_1,\dots,k_m)\in\mathbb Z^m$, let $\psi_g(k)$ be the diffeomorphism
  which is $f^{2k_j}$ on the support of $f$ in $s(i,j)(\partial C)$ for all $j\in\finset{m}$
  and the identity otherwise.
  Note that the exponent is doubled.
  The map $\psi_g:\mathbb Z^m\rightarrow\Diff(\partial\mathbb T)$ is a group homomorphism.
  Composing the quotient map $\Diff(\partial\mathbb T)\rightarrow\mathcal D(\partial\mathbb T)$,
  we get $\psi_g:\mathbb Z^m\rightarrow\mathcal D(\partial\mathbb T)$;
  we use the same symbol again.
  Since $\phi^e$ and $\psi_g$ for $g\in G$ satisfy the assumptions in Proposition \ref{homom_prop},
  we get the homomorphism $\omega:\mathbb Z^m\wr G\rightarrow\mathcal D(\partial\mathbb T)$.

  The inclusion map $\iota:\mathbb T\rightarrow Y$ is a $\mathbb Z^m\wr G$-effective embedding
  since we have $Y^\iota_{(F,g)}\cong X(2\tilde F+\delta_{(\widehat\phi(g)p,1)})$ for any $(F,g)\in\mathbb Z^m\wr G$,
  where $\tilde F\in\mathbb Z^{\finset{n}\times\finset{m}}$ is the element
  defined by $\tilde F(\widehat\phi(x)p,j)=F(x)(j)$ for any $x\in G$ and $j\in\finset{m}$
  and $\tilde F(i,j)=0$ otherwise.
  Thus $(\mathbb T,\omega)$ is a weakly equivariant $\mathbb Z^m\wr G$-cork.

  The latter is proved in the same way, and it is exactly the trick by \AKMR{} in \cite{auckly2017}.
  Replace Theorem \ref{emb_of_z_cork} with Theorem \ref{emb_of_2_cork} and
  Gompf's cork with the cork in Theorem \ref{emb_of_2_cork},
  and ignore the $\mathbb Z^m$-actions $\psi_g$ in the above proof.
\end{proof}

\begin{lem}
  \label{wreath_lemma}
  Let $G$, $H$ be groups.
  Assume that there exist two blocks $\Phi=(A;$ $e_1,\dots,e_l;$ $\epsilon)$ and $\Psi=(A';$ $e_1',\dots,e_m';$ $\epsilon')$ and
  two homomorphisms $\phi:G\rightarrow\mathcal D(\Phi)$ and $\psi:H\rightarrow\mathcal D(\Psi)$
  such that $\phi$ is (P2) and $\psi$ is (P1).
  Then there exist a block $\Omega=(\bar A;\bar e_1,\dots,\bar e_n;\bar\epsilon)$
  and a homomorphism $\omega:H\wr G\rightarrow\mathcal D(\Omega)$ such that $\omega$ is (P1).
\end{lem}

\begin{proof}
  We first glue $\Phi$ and $l$ copies of $\Psi$ in the following way:
  for any $i\in\finset{l}$, we glue $e_i$ in the block $\Phi$ and $\epsilon'$ in the $i$-th copy $\Psi_i$ of $\Psi$.
  The resulting manifold is denoted by $\bar A$.
  Let $\bar e_{i,j}$ be $e_j'$ in $\Psi_i$ for $(i,j)\in\finset{l}\times\finset{m}$ and let $\bar\epsilon$ be $\epsilon$ in $\Phi$.
  We define a block $\Omega:=(\bar A;\dots,\bar e_{i,j},\dots;\bar\epsilon)$.
  See Figure \ref{fig:omega}.

  \begin{figure}
  \centering
  \includegraphics[width=7cm]{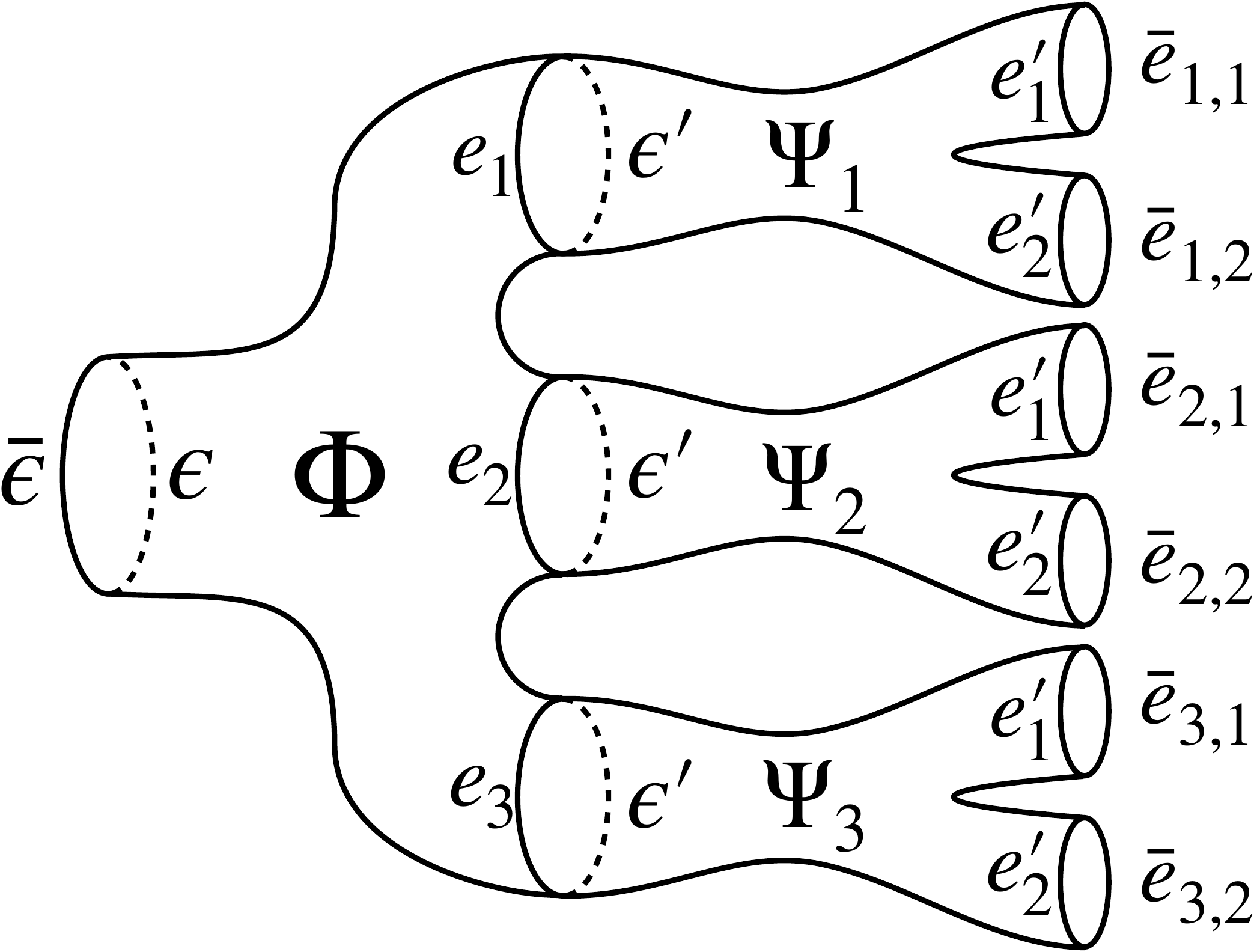}
  \caption{The block $\Omega$ in the case of $l=3,m=2$.}
  \label{fig:omega}
  \end{figure}

  We extend $\phi:G\rightarrow\mathcal D(\Phi)$ to $\phi:G\rightarrow\mathcal D(\Omega)$
  defined to be the permutation of $\Psi_i$ by $\widehat\phi$ on the rest.
  Let $p\in\finset{l}$ be a point as in (P2).
  Fix $g\in G$ and let $i:=\widehat\phi(g)p$.
  We extend $\psi:H\rightarrow\mathcal D(\Psi_i)$ on the $i$-th copy of $\Psi$
  to $\psi_g:H\rightarrow\mathcal D(\Omega)$ which is the identity on the rest.
  These $\phi$ and $\psi_g$ for $g\in G$ satisfy the assumptions in Proposition \ref{homom_prop}
  and we get the homomorphism $\omega:H\wr G\rightarrow\mathcal D(\Omega)$.
  It is easy to see that $\omega$ is (P1).
\end{proof}

\begin{lem}
  \label{zn_case}
  For any $n\ge 2$, there exist a block $\Phi=(A;e_1,\dots,e_{n^2};\epsilon)$
  and a homomorphism $\phi:\mathbb Z_n\rightarrow\mathcal D(\Phi)$ with (P2).
\end{lem}

\begin{proof}
  We first define a block $\Psi=(A';e_1',\dots,e_n';\epsilon')$.
  Let $A'$ be $B^2\times [0,1]\times [0,1]$,
  let $\epsilon'$ be a ball whose center is $(0,0,1/2)$,
  and let $e_k'$ be one whose center is $(\frac12e^{2\pi ik/n},1,\frac12)$ for $1\le k\le n$. 
  We define $f\in\Diff(\Psi)$ by $f(x,s,t):=(e^{2\pi is/n}x,s,t)$,
  which is the $\frac {2\pi s}n$-rotation of $B^2\times\{s\}\times\{t\}$ for $(s,t)\in[0,1]\times[0,1]$.

  Next, we connect $\Psi$ with $n$ copies $\Psi_1,\dots,\Psi_n$ of $\Psi$.
  We glue $e_i'$ in $\Psi$ and $\epsilon'$ in $\Psi_i$ for $i\in\finset{n}$.
  The resulting manifold is denoted by $A$.
  Let $e_{i,j}$ be $e_j$ in $\Psi_i$ and $\epsilon$ be $\epsilon'$ in $\Psi$.
  We combine them to get the block $\Phi:=(A;\dots e_{i,j}\dots;\epsilon)$.
  See Figure \ref{fig:zn_case}.

  \begin{figure}
  \centering
  \includegraphics[width=7cm]{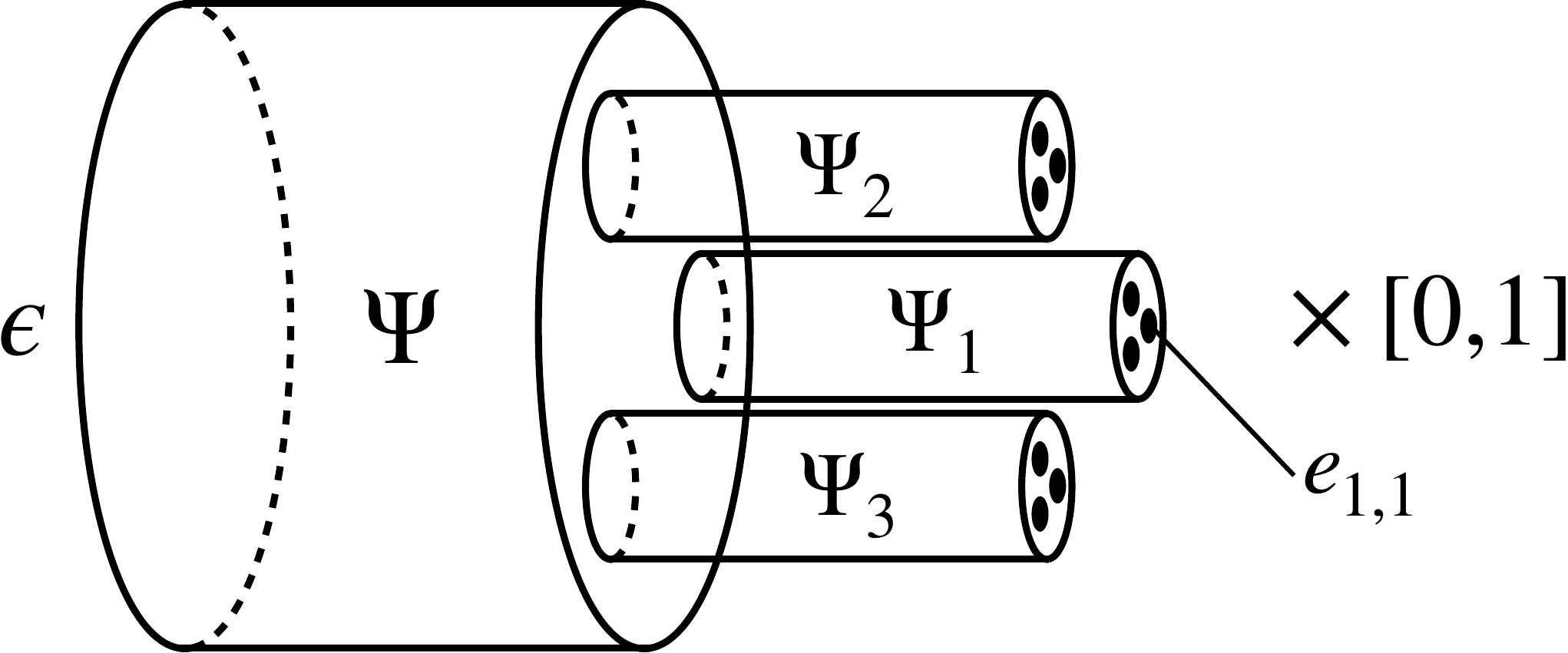}
  \caption{The block $\Phi$ in the case of $n=3$.}
  \label{fig:zn_case}
  \end{figure}
  
  For any $1\le i\le n$, let $f_i\in\Diff(\Phi)$ be the trivial extension of $f$ on $\Psi_i$,
  and let $f_0$ be the extension of $f$ on $\Psi$ defined to be the cyclic permutation of $\Psi_i$ on the rest.
  We define a homomorphism
  $\phi'':\mathbb Z\rightarrow\Diff(\Phi)$ by $\phi''(k):=f_1^kf_2^k\dots f_n^kf_0^k$.
  Let $\phi':\mathbb Z\rightarrow\mathcal D(\Phi)$ be the composition of $\phi''$ and the quotient map.

  Since $\phi'(n)$ is the Dehn twists along each $3$-ball $B^2\times\{1/2\}\times I$ in
  $\Psi$, $\Psi_1$, $\Psi_2$, $\dots$, $\Psi_n$
  and these $3$-balls can be pushed into the boundary $\partial A$ disjointly,
  we have $\phi'(n)=1$ in $\mathcal D(\Phi)$ by \cite[Lemma 3.1]{auckly2017}.
  Thus $\phi'$ factors through the quotient map $\mathbb Z\rightarrow\mathbb Z_n$ and
  we get $\phi:\mathbb Z_n\rightarrow\mathcal D(\Phi)$.
  It is easy to see that this satisfy (P2).
\end{proof}

\begin{lem}
  \label{p1_to_p2}
  If there exist a block $\Phi=(A;e_1,\dots,e_n;\epsilon)$
  and a homomorphism $\phi:G\rightarrow\mathcal D(\Phi)$ with (P1),
  there exist another block $\Psi=(A';e_1',\dots,e_m';\epsilon')$ and
  another homomorphism $\psi:G\rightarrow\mathcal D(\Psi)$ with (P2).
\end{lem}

\begin{proof}
  We combine $1+n+n^2+\dots+n^{n-1}$ copies of $\Phi$ to obtain $\Psi$.
  First, let $\Phi(1),\dots,\Phi(n)$ be $n$ copies of $\Phi$.
  (We use the notation $\Phi(i)$ to indicate an index rather than $\Phi_i$
  since indices will become complicated.)
  We glue $e_i$ in $\Phi$ and $\epsilon$ in $\Phi(i)$ for all $1\le i\le n$.
  Next, for $1\le i,j\le n$, let $\Phi(i,j)$ be a copy of $\Phi$.
  Again we glue $e_j$ in $\Phi(i)$ and $\epsilon$ in $\Phi(i,j)$ for all $1\le i,j\le n$.
  Repeat this process $n-1$ times.
  Let $A'$ be the resulting manifold,
  let $e'(i_1,\dots,i_n)$ be $e_{i_n}$ in $\Phi(i_1,\dots,i_{n-1})$ for $(i_1,\dots,i_n)\in\finset{n}^n$,
  and let $\epsilon'$ be $\epsilon$ in $\Phi$.
  We define a block $\Psi:=(A';\dots e'(i_1,\dots,i_n)\dots;\epsilon')$.
  See Figure \ref{fig:p1_to_p2}.

  \begin{figure}
  \centering
  \includegraphics[width=8cm]{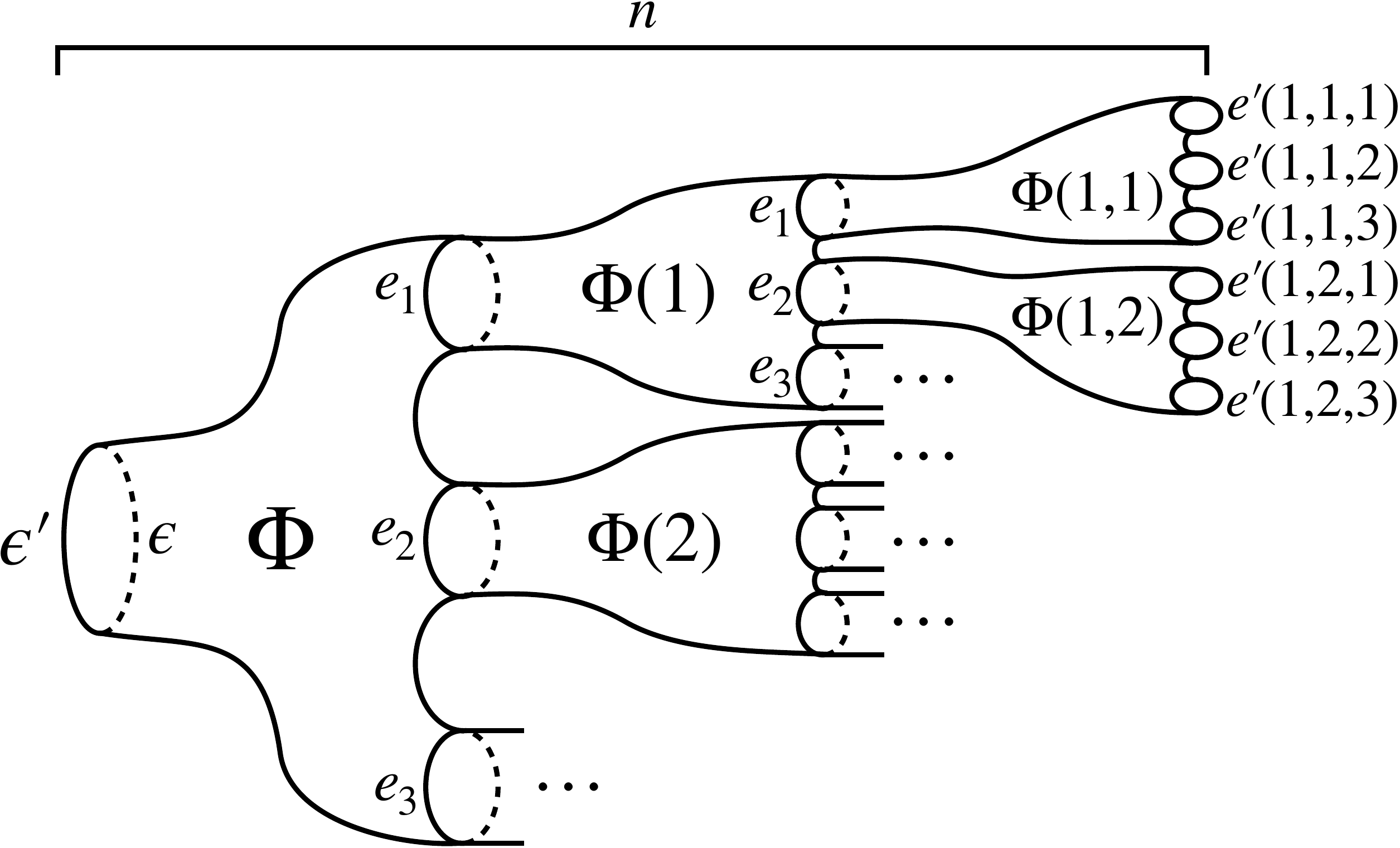}
  \caption{The block $\Psi$ in the case of $n=3$.}
  \label{fig:p1_to_p2}
  \end{figure}

  For any $0\le r\le n-1$ and any $1\le i_1,i_2,\dots,i_r\le n$,
  we extend $\phi:G\rightarrow\mathcal D(\Phi(i_1,\dots,i_r))$
  to $\phi(i_1,\dots,i_r):G\rightarrow\mathcal D(\Psi)$
  which is the permutation of the components joined to $e_1,\dots,e_n$ by $\widehat\phi$ and
  the identity on the one joined to $\epsilon$. 
  We define a homomorphism $\psi:G\rightarrow\mathcal D(\Psi)$ by
  \[
  \psi(g):=\prod_{0\le r\le n-1}\left(\prod_{1\le i_1,i_2,\dots,i_r\le n}\psi(i_1,i_2,\dots,i_r)(g)\right).
  \]
  Note that $\psi$ does not depend on the order of multiplication.

  Since we have $\widehat\psi(g)(1,2,\dots,n)=(\widehat\phi(g)1,\widehat\phi(g)2,\dots,\widehat\phi(g)n)$
  and $\phi$ is (P1), $\psi$ is (P2) 
\end{proof}

\begin{proof}[Proof of Theorem \ref{thm:wreath:weak} \ref{thm:wreath:weak:weak} and \ref{thm:wreath:weak:Stein}]
  Let $m\ge1$ and $n_1,\dots,n_r\ge2$ be integers.
  By Lemma \ref{zn_case}, there exist a block $\Phi^{(1)}_i$
  and a (P2) homomorphism $\phi^{(1)}_i:\mathbb Z_{n_i}\rightarrow\mathcal D(\Phi^{(1)}_i)$.
  Applying Lemma \ref{wreath_lemma} repeatedly, we get a block $\Phi^{(2)}$ and
  a (P1) homomorphism $\phi^{(2)}:\mathbb Z_{n_1}\wr\dots\wr\mathbb Z_{n_r}\rightarrow\mathcal D(\Phi^{(2)})$.
  Lemma \ref{p1_to_p2} implies that there exist a block $\Phi^{(3)}$
  and a (P2) homomorphism $\phi^{(3)}:\mathbb Z_{n_1}\wr\dots\wr\mathbb Z_{n_r}\rightarrow\mathcal D(\Phi^{(3)})$.
  We get a weakly equivariant $\mathbb Z^m\wr(\mathbb Z_{n_1}\wr\dots\wr\mathbb Z_{n_r})$-cork
  and a weakly equivariant Stein $\mathbb Z_{n_1}\wr\dots\wr\mathbb Z_{n_r}$-cork by Lemma \ref{weak_cork_lem}.
\end{proof}

\begin{que}
  Does there exists a weakly equivariant (Stein) $G$-cork for any finite (or even countable) group $G$?
\end{que}

\begin{que}
  Are Theorem \ref{thm:main} \ref{thm:main:weak} and \ref{thm:main:weakStein} true
  if we assume a weakly equivariant $G$-cork to be boundary sum irreducible?
\end{que}

\section{Applications to diffeotopy groups of exotic $\mathbb R^4$'s}
\label{sec:diff}
Gompf recently applied cork theory and constructed exotic $\mathbb R^4$'s
with large diffeotopy groups in \cite{gompf2018}.
We can also use the wreath product approach in this context and get the following theorem.
Hereafter, we define $\mathbb Z_0:=\mathbb Z$.
\begin{thm}
  \label{thm:app:wr}
  There exists an exotic $\mathbb R^4$ $R$ such that there exists an embedding
  $\mathbb Z_{n_1}\wr\dots\wr\mathbb Z_{n_r}\rightarrow\mathcal D(R)$
  for any $n_1,\dots,n_r\in\mathbb N\cup\{0\}$.
\end{thm}

\begin{rem}
  \label{rem_of_d}
  \begin{tightenum}
    \renewcommand{\theenumi}{(\alph{enumi})}
    \renewcommand{\labelenumi}{\theenumi}
  \item
    \label{gompf_said}
    Gompf already mentioned the key point of the proof of the theorem in \cite[the last sentence of Example 4.3 (a)]{gompf2018}.
    The author specified the groups in it
    and used weak actions
    \footnote{A weak action indicates a homomorphism into the diffeotopy group of a manifold.}
    in \cite[Theorem 4.4]{gompf2018}.
  \item
    Since the resulting manifold of the proof of the theorem is same as that of \cite[Theorem 4.4]{gompf2018},
    we can assume that $R$ meets same conditions as \cite[Theorem 4.4]{gompf2018}.
  \end{tightenum}
\end{rem}

We see subgroups of $\mathbb Z_{n_1}\wr\dots\wr\mathbb Z_{n_r}$.

\begin{defn}[{\cite[p.122]{suzuki}}]
  \label{def_of_pc}
  A group $G$ is said to be poly-cyclic if there exists a subnormal series
  $G=G_0\triangleright G_1\triangleright\dots\triangleright G_r={1}$
  such that each quotient group $G_{i-1}/G_i$ is a (finite or infinite) cyclic group.
\end{defn}

\noindent A finite solvable group is poly-cyclic.
We have the following theorem in the same way as Corollary \ref{kk_cor_fs}.

\begin{cor}
  \label{kk_cor_pc}
  For any poly-cyclic group $G$, there exist natural numbers $p_1,\dots,p_r$ which are prime or $0$
  such that $G$ can be embedded into $\mathbb Z_{p_1}\wr\dots\wr\mathbb Z_{p_r}$.
\end{cor}

\noindent Thus, $\mathcal D(R)$ in Theorem \ref{thm:app:wr} contains any poly-cyclic group
and we get Theorem \ref{thm:app:pc}.

The proof of Theorem \ref{thm:app:wr} is very similar to that of \cite[Theorem 4.4]{gompf2018}.
While Gompf mainly considered group actions,
we use weak group actions.

First, we consider an open analog of a block in Section \ref{sec:weak_corks}.
An \textit{open block} $\Phi=(A;$ $\{e_\lambda\}_{\lambda\in\Lambda};$ $\epsilon)$ is the tuple of an open manifold $A$
diffeomorphic to $\mathbb R^4$ and disjoint embeddings
$e_{\lambda},\epsilon:[0,\infty)\times B^3\rightarrow A$ each of which is proper.
Since we consider open manifolds now, $\Lambda$ can be countably infinite.
Note that we do not assume that the disjoint union of $e_\lambda$ for $\lambda\in\Lambda$ and $\epsilon$ is proper,
i.e., the embedding
\begin{equation}
\label{du_of_e}
\bigsqcup_{\lambda\in\Lambda}e_\lambda\sqcup\epsilon:\bigsqcup_{\lambda\in\Lambda}
\left([0,\infty)\times B^3\right)\sqcup\left([0,\infty)\times B^3\right)\rightarrow A
\end{equation}
can be non-proper.    

We define $\Diff(\Phi)$, $\mathcal D(\Phi)$, \ $\widehat\cdot :\mathcal D(\Phi)\rightarrow \mathfrak S(\Lambda)$,
and the (P1) and (P2) properties in the same way,
except we use the symmetric group $\mathfrak S(\Lambda)$ on the set $\Lambda$ instead of $\mathfrak S_n$.
To be more precise, a diffeomorphism $f$ of $A$ is in $\Diff(\Phi)$
if and only if we have $f\circ\epsilon=\epsilon$ and there exists $\sigma\in\mathfrak S(\Lambda)$ with
$f\circ e_\lambda=e_{\sigma(\lambda)}$ for all $\lambda\in\Lambda$,
$\mathcal D(\Phi)$ is the quotient group of $\Diff(\Phi)$ by diffeotopies through it,
and we define $\widehat{[f]}:=\sigma$.

\begin{defn}
  We define $\gsw(i)$ to be the set of isomorphic classes of groups
  each of which has a (P$i$) homomorphism into $\mathcal D(\Phi)$ for some open block $\Phi$,
  where $i=1,2$.
\end{defn}
    
\noindent The set $\gsw(i)$ is an analog of $\mathcal G^*$ in \cite{gompf2018} with weak actions.
Note that the (P2) condition was used in \cite[Theorem 4.6]{gompf2018}.

\begin{thm}
  \label{gompf_main}
  There exists an exotic $\mathbb R^4$ R whose diffeotopy group $\mathcal D(R)$ contains
  any group in $\gsw(1)$.
\end{thm}

\noindent This theorem can be shown in the completely same way as \cite[Theorem 4.4]{gompf2018}.

We then prove an open analog of Lemma \ref{wreath_lemma}.
This is what Gompf already mentioned implicitly (see Remark \ref{rem_of_d} \ref{gompf_said}).
\begin{lem}
  \label{open_wreath_lemma}
  For any $H\in\gsw(1)$ and $G\in\gsw(2)$, we have $H\wr G\in\gsw(1)$.
\end{lem}

\begin{proof}
  The proof is same as that of Lemma \ref{wreath_lemma},
  except we need Proposition \ref{homom_prop} with infinite $H$.

  Let $\Phi=(A;\{e_\lambda\}_{\lambda\in\Lambda};\epsilon)$
  and $\Psi=(A';\{e_{\mu}'\}_{\mu\in M};\epsilon')$ be open blocks
  and let $\phi:G\rightarrow\mathcal D(\Phi)$ and $\psi:H\rightarrow\mathcal D(\Psi)$
  be homomorphisms with (P2) and (P1), respectively.
  We connect $\Phi$ with copies $\Psi_\lambda$ of $\Psi$ indexed by $\lambda\in\Lambda$
  as in the proof of Lemma \ref{wreath_lemma},
  using end summing instead of boundary summing.
  The resulting manifold is the end sum $A\natural_{\lambda\in\Lambda}A'$
  of finite or countably infinite copies of $\mathbb R^4$ and thus diffeomorphic to $\mathbb R^4$
  (see \cite[the proof of Theorem 4.4]{gompf2018}
  for how to deal with non-properness of the map (\ref{du_of_e})).
  Let $e_{\lambda,\mu}$ be $e_\mu'$ in $\Psi_\lambda$.
  We get the open block
  $\Omega:=(A\natural_{\lambda\in\Lambda}A';\{e_{\lambda,\mu}\}_{(\lambda,\mu)\in\Lambda\times M};\epsilon)$.

  The homomorphism $\phi:G\rightarrow\mathcal D(\Phi)$ extends to $\phi:G\rightarrow\mathcal D(\Omega)$ naturally.
  Let $p\in\Lambda$ be a point as in (P2).
  We define $\chi(F)\in\mathcal D(\Omega)$ for $F\in H^G$
  to be $\psi(F(g))$ on $\Psi_{\widehat\phi(g)p}$ for all $g\in G$ and the identity on the rest.
  The map $\omega:H\wr G\rightarrow\mathcal D(\Omega)$
  defined by $\omega(F,h):=\chi(F)\phi(h)$ is a homomorphism with (P1).
\end{proof}

Lastly, we prove an open analog of Lemma \ref{zn_case}.

\begin{lem}
  \label{open_zn_case}
  For any $n\ge0$ with $n\neq1$, we have $\mathbb Z_n\in\gsw(2)$.
\end{lem}
\noindent Allowing manifolds to be open implies the case of $n=0$.
\begin{proof}
  The case of $n\ge2$ follows from Lemma \ref{zn_case}.
  Thus we consider the case of $n=0$.
  Let $f:\mathbb R\rightarrow[0,1]$ be a smooth function
  with $f=0$ on $(-\infty,-1]$ and $f=1$ on $[0,\infty)$.
  Let $e_i$ be the tubular neighborhood of $[0,\infty)\times\{i\}\times\{0\}\times\{0\}$ in $\mathbb R^4$
  and let $\epsilon$ be that of $(-\infty,-1]\times\{0\}\times\{0\}\times\{0\}$ in it.
  Then the tuple $\Phi:=(\mathbb R^4;\{e_i\}_{i\in\mathbb Z};\epsilon)$ is an open block.
  The diffeomorphism $g$ of $\Phi$ defined by $g(x,y,z,w):=(x,f(x)+y,z,w)$
  generates a (P2) homomorphism $\mathbb Z\rightarrow\mathcal D(\Phi)$.
\end{proof}

\begin{proof}[Proof of Theorem \ref{thm:app:wr}.]
  For any $n_1,\dots,n_r\in\mathbb N\cup\{0\}$,
  Lemma \ref{open_zn_case} and Lemma \ref{open_wreath_lemma} imply $\mathbb Z_{n_1}\wr\dots\wr\mathbb Z_{n_r}\in\gsw(1)$
  and thus the theorem follows from Theorem \ref{gompf_main}.
\end{proof}

\begin{rem}
  An analog of \cite[Theorem 4.6]{gompf2018} is also true,
  i.e., there exists an exotic $\mathbb R^4$ $R$ such that every group $G\in\gsw(2)$ has an embedding $G\rightarrow\mathcal D^\infty(R)$
  with $G\cap r(\mathcal D(R))=\{1\}$.
  See \cite{gompf2018} for details.
\end{rem}

\begin{que}
  Does there exist an exotic $\mathbb R^4$ whose diffeotopy group contains any finite (or even countable) group?
  If this is false, for any finite (or countable) group $G$,
  does there exist an exotic $\mathbb R^4$ whose diffeotopy group contains $G$?
\end{que}

\section*{Acknowledgements}
The author would like to thank his advisor Kenta Hayano for encouragement and useful comments.

\bibliographystyle{plain}
\bibliography{ref}

\end{document}